\documentclass[12pt]{amsart}

\usepackage{amssymb}
\usepackage{setspace}
\usepackage{xy}
\usepackage[usenames, dvipsnames, svgnames]{xcolor}

\usepackage[colorlinks=true,linkcolor=blue,urlcolor=blue,citecolor=blue, pagebackref]{hyperref}

\usepackage{enumerate} 

\usepackage{helvet} 
\usepackage{courier} 
\usepackage{eucal} 
\usepackage{mathrsfs} 

\reversemarginpar 
\normalmarginpar 

\let\oldmarginpar\marginpar 
\renewcommand\marginpar[1]{\-\oldmarginpar{\raggedright\small\sf #1}}

\newcommand{\nc}{\newcommand}

\nc{\rnc}{\renewcommand}

\nc{\bs}{\backslash}
\nc{\te}{\otimes}
\nc{\lf}{\lfloor} 
\nc{\rf}{\rfloor}
\nc{\lc}{\lceil}  
\nc{\rc}{\rceil}
\nc{\lr}{\longrightarrow}
\nc{\sr}{\stackrel}
\nc{\dar}{\dashrightarrow}
\nc{\thra}{\twoheadrightarrow}

\nc{\la}{\langle}
\nc{\ra}{\rangle} 

\nc{\ms}{\mathscr}
\nc{\mc}{\mathcal}
\nc{\mb}{\mathbb}
\nc{\mf}{\mathbf}
\nc{\mr}{\mathrm}
\nc{\mg}{\mathfrak}

\nc{\bP}{\mathbb{P}}
\rnc{\P}{\mathbb{P}}
\nc{\Q}{\mathbb{Q}}
\nc{\Z}{\mathbb{Z}}
\nc{\C}{\mathbb{C}}
\nc{\R}{\mathbb{R}}
\nc{\A}{\mathbb{A}}
\nc{\V}{\mathbb{V}}
\nc{\W}{\mathbb{W}}
\nc{\N}{\mathbb{N}}
\nc{\D}{\mathbb{D}}
\nc{\G}{\mathbb{G}}
\nc{\F}{\mathbb{F}}
\nc{\qb}{\overline{\mathbb{Q}}}

\nc{\del}{\partial}

\nc{\wt}{\widetilde}
\nc{\wh}{\widehat}
\nc{\ov}{\overline}
\nc{\un}{\underline}

\nc{\naive}{\!\sim_n}

\nc{\Spec}{\mr{Spec}}

\nc{\omx}{\omega_X}

\nc{\ep}{\epsilon}
\nc{\ve}{\varepsilon}
\nc{\vt}{\vartheta}
\nc{\rhobar}{\overline{\rho}}
\rnc{\l}{\lambda}
\rnc{\k}{\kappa}

\nc{\ovl}{\ov{\lambda}}
\nc{\vl}{\mb{V}_{\ovl}}
\nc{\dl}{\mb{D}_{\ovl}}
\nc{\mnb}{\ov{\mr{M}}_{0,n}}
\nc{\mn}{\mr{M}_{0,n}}
\nc{\mel}{\ov{\mr{M}}_{1,1}}
\nc{\mfb}{\ov{\mr{M}}_{0,4}}
\nc{\mof}{\mr{M}_{0,4}}
\nc{\mgnb}{\ov{\mr{M}}_{g,n}}
\nc{\mgn}{\ov{\mr{M}}_{g,n}}
\nc{\omc}{\ov{\mr{M}}}

\rnc{\sl}{\shoveleft}

\nc{\res}{\operatorname{Res}}
\nc{\pic}{\operatorname{Pic}}
\nc{\spec}{\operatorname{Spec}}
\nc{\im}{\operatorname{Im}}
\nc{\gal}{\operatorname{Gal}}
\nc{\fr}{\operatorname{Fr}}
\nc{\ed}{\operatorname{ed}}
\nc{\rank}{\operatorname{rank}}
\nc{\h}{\operatorname{H}}
\nc{\ch}{\operatorname{char}}
\nc{\sw}{\operatorname{sw}}
\nc{\rsw}{\operatorname{rsw}}
\nc{\supp}{\operatorname{supp}}
\nc{\id}{\operatorname{id}}
\nc{\Ad}{\operatorname{Ad}}
\nc{\cO}{\mathcal{O}}
\nc{\Mor}{\operatorname{Mor}}
\nc{\Per}{\operatorname{Per}}
\nc{\prep}{\operatorname{Prep}}
\nc{\End}{\operatorname{End}}
\nc{\Orb}{\operatorname{Orb}}
\nc{\Aut}{\operatorname{Aut}}
\nc{\tr}{\operatorname{Tr}}
\nc{\GL}{\operatorname{GL}}
\nc{\SL}{\operatorname{SL}}
\nc{\Frob}{\operatorname{Frob}}
\nc{\Br}{\operatorname{Br}}
\nc{\inv}{\operatorname{inv}}
\nc{\chr}{\operatorname{char}}
\nc{\br}{\bar{\rho}}
\nc{\ideala}{\mathfrak{a}}
\nc{\m}{\mathfrak{m}}
\nc{\primep}{\mathfrak{p}}
\nc{\primeq}{\mathfrak{q}}
\renewcommand{\sl}{\mathfrak{sl}}

\newtheorem{thm}{Theorem}[section]
\newtheorem{prop}[thm]{Proposition}
\newtheorem{conj}[thm]{Conjecture}
\newtheorem{cor}[thm]{Corollary}
\newtheorem{lem}[thm]{Lemma}
\newtheorem{quest}[thm]{Question}

\newtheorem{guess}[thm]{Guess}

\newtheorem{thmalph}{Theorem}

\theoremstyle{definition}
\newtheorem{defn}[thm]{Definition}

\newtheorem{rem}[thm]{Remark}

\newtheorem{define}[thm]{Definition}

\numberwithin{equation}{section}

\input{xypic}

\title{Liftable  groups, negligible cohomology and Heisenberg representations }

\begin{document}

\author[C. B. ~Khare]{Chandrashekhar  B.  Khare}
\address{UCLA Department of Mathematics, Box 951555, Los Angeles, CA 90095, USA}
\email{shekhar@math.ucla.edu}
\author[Michael Larsen]{Michael Larsen}
\address{Department of Mathematics, Indiana University, Bloomington, IL}
\email{mjlarsen@indiana.edu}


\maketitle

\begin{abstract}

We consider lifting of mod $p$ representations to mod $p^2$ representations  in the setting of representations of
(i)  finite groups; (ii) absolute Galois groups of abstract fields; and (iii) absolute Galois groups of local and global fields.

\end{abstract}

\section{Introduction}

For any field $L$, denote  by $L^s$ a separable closure of $L$, and  let $G_L={\rm Gal}(L^s/L)$ denote the absolute Galois group of $L$.
In \cite{K-JNT} it was  proved that for any field $L$,  any continuous homomorphism $\rho:G_L \rightarrow \GL_2(k)$ with $k$ a finite field  lifts to $\GL_2(W_2(k))$. This was originally proved  (cf.  \cite[Theorem 1]{K-JNT}) with the assumption that $L$ is a number field, and Serre pointed out at the time  \cite[Remark 2, pg. 392]{K-JNT}   that  the argument  uses only Kummer theory, and thus  being purely algebraic, actually works for {\it all} fields $L$ (see Proposition \ref{lifting n=2}).

The motivation for \cite{K-JNT} was Serre's modularity conjecture \cite{serre:conjectures} which implies that irreducible odd representations $\br:G_\Q \rightarrow \GL_2(k)$ lift to {\it geometric} representations in characteristic 0 (see \S \ref{app}).  We recall that $\rhobar$ is called odd if $\det(\rhobar(c))=-1$ where $c \in G_\Q$ is a complex conjugation. A continuous representation  $\rho:G_K \to \GL_n(\overline \Q_p)$  for a number field $K$ is said to be {\it geometric}   if it is unramified outside a finite set of primes, and is de Rham at all places of $K$ above $p$.

Lifting methods have been greatly developed over number fields following:

--  a method of Ramakrishna \cite{ramakrishna02} which  uses methods of algebraic number theory (duality theorems for Galois cohomology and the Chebotarev density theorem ),

-- and  another method in \cite{khare-wintenberger:serre0} (which uses the modularity lifting technique of Wiles \cite{wiles:fermat}  and its developments).  

(We discuss this further  in Appendix A, where we give references to later developments of these methods.)
The latter method was instrumental in the proof of Serre's modularity conjecture in  \cite{khare-wintenberger:serre1}. Both these  techniques  have in common that they use local-global methods and are more sophisticated than the  method used to get mod $p^2$ liftings constructed in \cite{K-JNT}. 

The mod $p^2$ lifting result of \cite{K-JNT} for 2-dimensional  mod $p$ representations  raises the question   if lifting Galois representations works in generality for $n$-dimensional  representations  over abstract fields.  The local-global  methods of \cite{ramakrishna02}  do not cut any ice here, and if one  seeks  to generalize the  mod $p^2$ lifting result  of \cite{K-JNT} one  has to  devise  arguments which exploit  results available over all fields like Kummer theory.  We note that lifting mod $p^2$ of $n$-dimensional mod $p$  representations of $G_L$  has been explored in  \cite{bockle:modp2}  in the context of $L$  a local or global field, and more recently in the   preprint \cite{mfcdc} for general  fields $L$.

\subsection{Questions}
This paper considers  mod $p^2$ lifting questions for finite groups, then Galois groups of abstract fields, and  then in the number theoretic case of local and global fields. We  study a  liftability property for groups $G$.  Namely $G$ is  said to be $L$ (resp. $L_p$) or  ($p$-) liftable, if any homomorphism $f:G \rightarrow \GL_n(k)$   with $k$ a finite field (of characteristic $p$) lifts to $\GL_n(W_2(k))$.

\begin{quest}\label{galois}
 Are all absolute Galois groups  of fields  liftable?
\end{quest}

The only result we know towards  this question are lifting results for homomorphisms  $f:G_L \rightarrow \GL_n(k)$ for $n \leq 2$ (Proposition \ref{lifting n=2}).

By a classical result of Artin and Schreier, the only finite groups which are absolute  Galois groups are $G=1,\Z/2\Z$, and these are liftable. In the context of  Question \ref{galois}  above, N. Fakhruddin asked if one can characterize finite groups $G$ that are $L$. 

 \begin{quest}(Fakhruddin)\label{finite groups}
Which finite groups  $G$ are liftable?  
\end{quest}

We tentatively  make the following guesses.

\begin{guess}
\label{local}
 If $G$ is a finite group whose order is divisible by a prime $p>3$, then $G$ is not $L_p$, and hence not $L$.  A $3$-group $G$ is $L_3$ if and only if $G\cong \Z/3\Z$.  A $2$-group $G$ is $L_2$ if and only if it is cyclic.
 \end{guess}

\begin{guess}
\label{strong}
A finite group $G$ is  liftable if and only if it is isomorphic to one of  $\Z/2^n\Z$, $\Z/3\Z\times\Z/2^n\Z$, or $\Z/3\Z\rtimes \Z/2^n\Z$.
\end{guess}

We justify Guess \ref{local} below for odd primes $p$ (see Corollary \ref{classification}).
Our evidence does not justify upgrading these guesses to conjectures, as   the guess characterizing $2$-groups  that are $L_2$ seems  open to doubt. We show that a 2-group is $L_2$ implies that  it is either cyclic or quaternionic.
But we have not so far succeeded in establishing even for $G=Q_8$, the quaternionic group of order 8, whether it is $L_2$ or not.

\subsection{Structure of the paper}

We begin the paper by studying  Question \ref{finite groups} in \S \ref{Finite Groups}. The groups  $1,\Z/2\Z,\Z/3\Z$ are  liftable, and $\Z/p\Z$ is not liftable for $p>3$ (see Proposition \ref{cyclic p-groups}).  One would expect that for a finite group
the property of being liftable  is very restrictive. To our surprise it turns out that $\Z/2^n\Z$ is liftable for all $n \geq 0$ (see Proposition \ref{power of 2}). 
Hence any  finite group $G$ is  liftable whose only Sylow subgroups are cyclic 2-groups and groups of order dividing $3$ (see Proposition \ref {reduction to p groups}).    In particular $S_3$ is liftable.

Guess  \ref{local} implies Guess \ref{strong} (see Proposition \ref{implication}), and we prove 
the guesses when we restrict to  finite abelian groups  (see Proposition \ref{abelian}). Using an observation of Ali Cheraghi   (Lemma \ref{ali} below) we deduce  in Corollary \ref{classification}   our Guess \ref{local} for $p>2$, and shows  in the case $p=2$ that $G$  has to be either cyclic or  a quaternionic group.  In \S \ref{Rigidity} we explore a certain rigidity  property of finite groups $G$ {\it relative} to a homomorphism   $G \to \GL_n(k)$, namely that any lifting of this homomorphism to $\GL_n(A)$ implies that $pA=0$. In  \S \ref{Finite Groups} we   are thus  in the realm of finite group theory.

In \S \ref{Negligible} we address Question \ref{galois} by recalling the results of \cite{K-JNT} on lifting 2-dimensional representations, and thus we are still in the realm of algebra, working with Galois groups of abstract fields.   In the 3-dimensional case, under the further assumption that $\mu_{p^2} \subset L$, we reduce  the problem of lifting representations $\rhobar:G_L \to \GL_3(\F_p)$ to $\GL_3(\Z/p^2\Z)$  (cf. Lemma \ref{red} and  \ref{red1}) to answering the following question:

\begin{quest}\label{quest:cup-product}
 Let $x_1,x_2 \in H^1(G_L,\mu_p)$ be orthogonal elements  under the cup-product map $\cup: H^1(G_L,\mu_p) \times H^1(G_L,\mu_p) \to H^1(G_L,\mu_p^{\otimes 2})$, i.e.,  $x_1 \cup x_2=0$. Do $x_i$ lift to elements $\tilde x_i \in H^1(G_L,\mu_{p^2})$,   such that  under the cup-product map $\cup: H^1(G_L,\mu_{p^2}) \times H^1(G_L,\mu_{p^2}) \to H^1(G_L,\mu_{p^2}^{\otimes 2})$,
 $\tilde x_i \cup \tilde x_2=0$?
\end{quest}

Being unable to answer this question in the generality of abstract  fields $L$, in \S \ref{cupping} we   transition from abstract fields to local and global fields, and address the above  question in this setting. In \S \ref{Heisenberg}  we address  lifting 
mod $p$ Heisenberg representations of Galois groups of such fields (with image the Heisenberg group of order $p^3$ in $\GL_3(\F_p)$)  to a  mod $p^2$  Heisenberg representation (cf. Definition \ref{def:Heisenberg}) with image in the upper triangular  unipotent matrices of $\GL_3(\Z/p^2\Z)$.   The main results  of \S \ref{cupping}   are Proposition \ref{solve local}, Theorem \ref{key}.

  \begin{thmalph}[see  Proposition \ref{solve local} and Theorem \ref{key}]
 Let $L$ be a  local non-archimedean field or global   field,  $p>2$ a prime, and  $\mu_p \subset L$.   Let  $x_1,x_2 \in H^1(L,\mu_p)$ satisfy  $x_1 \cup x_2=0$.
 %
%
Then there exist elements $\tilde x_1,\tilde x_2\in H^1(L,\mu_{p^2})$ mapping  to $x_1$ and $x_2$ respectively such      
that $\tilde x_1\cup \tilde x_2 = 0$.
    \end{thmalph}

We deduce from this the following theorem.  By {\it unitriangular} matrices in $\GL_n(R)$  for a ring $R$ we mean upper triangular matrices with 1's on the diagonal. By a unitriangular representation of a group to $\GL_n(R)$   we mean one whose image is contained in the  unitriangular matrices of $\GL_n(R)$

\begin{thmalph}[See Theorem \ref{final}]\label{mainthmintro}

 Let $L$ be a local non-archimedean field or a global field, $p>2 $ a prime, and assume $\mu_{p^2}  \subset L$.
 
 -- A continuous  representation $\rhobar:G_L \to \GL_3(\F_p)$ with image in the unitriangular matrices  of $\GL_3(\F_p)$  lifts to a  mod $p^2$  representation  $\rho:G_L\to  \GL_3(\Z/p^2\Z)$ which takes values in the unitriangular matrices of $\GL_3(\Z/p^2\Z)$.
  
 -- A continuous representation $\rhobar: G_L \to \GL_3(\F_p)$ lifts to a representation $\rho:G_L \to \GL_3(\Z/p^2\Z)$.

\end{thmalph}

  Theorem \ref{key} answers  Question \ref{quest:cup-product} in the case of $L$ a local or global field. In the 3-dimensional case for lifting to mod $p^2$ representations in Theorem \ref{final}   we take advantage of the fact that the obstruction to lifting from mod $p$ to mod $p^2$ representations  is related to the  cup product lifting property proved in Theorem \ref{key}.  

We observe that mod $p^2$  liftings in the 2-dimensional case are deduced from the surjectivity of  the map $H^1(L,\mu_{p^2}) \to H^1(L,\mu_p)$ that is a generality for all fields (Kummer theory). In the 3-dimensional case we need the surjectivity of the map  $H^1(L,\mu_{p^2}) \to H^1(L,\mu_p)$ together with being able to lift mod $p$ Kummer elements $x_1 , x_2$  to mod $p^2$ Kummer elements  with cup product that is a specified  lift of $x_1 \cup x_2$ to $H^2(L,\mu_{p^2}^{\otimes 2})$ (cf. Question \ref{quest:cup-product}). 


Our approach is specific to local and global fields as it  draws heavily  on duality theorems and  the Chebotarev density theorem, but we hope it might still give some hints  on how to address lifting problems as in   Question \ref{galois} (over abstract fields)   beyond the 2-dimensional case. 
In an appendix  \S \ref{app} we raise somewhat more specialized  questions of lifting of Galois representations over number fields.

 \vskip .5cm
 
 The paper has its origins in a talk one of us (CBK) gave at the TIFR   International Colloquium on Arithmetic Geometry in January 2020. There we  divided the talk into two parts:  ``Algebra'' and ``Number Theory''. The ``Algebra'' part recalled the results in \S \ref{Negligible} from \cite{K-JNT} and posed Question \ref{galois}. In the ``Number Theory''  part  we stated some results and consequences of the methods of \cite{fkp:reldef} and \cite{fkp:reducible}. The organization of the paper reflects that of  the talk, and  proceeds from algebra (results about finite groups in \S \ref{Finite Groups}, abstract Galois groups in \S \ref{Negligible}) to number theory (results in \S \ref{cupping} and \S  \ref{Heisenberg}).    The results of \S \ref{Finite Groups}, \S \ref{cupping} and \S \ref{Heisenberg} of this paper  are new, and were arrived at in  the period April--June 2020 during the pandemic.

As the first named author  recounted in his lecture at the conference,   he started thinking about questions related to lifting of Galois representations while  he worked at TIFR (for roughly  a decade)   upon   returning to India     after  completing his graduate studies in Los Angeles (at Caltech and UCLA) in 1995.  This was  motivated by Serre's conjecture \cite{serre:conjectures}  which implied existence of  liftings  of odd, irreducible $\rhobar:G_\Q \to \GL_2(k)$. The paper \cite{K-JNT} dates from that period. He  went on to  sheepishly confess to   still thinking about these questions twenty five years thence.   It is our hope that the present paper illustrates   that there are many  interesting questions about lifting Galois representations over abstract fields, and local and global fields,   that remain to be answered.

The first named author thanks  the organizers of the International Colloquium on Arithmetic Geometry at TIFR in January 2020  for  their invitation  to a stimulating and enjoyable  conference (leavened with concerts of  Indian classical  music and dance, keeping with the tradition of the International Colloquiums at TIFR!),  which led  to this paper. 

We would like to thank N. Fakhruddin for pointing out  Question \ref{finite groups}  (during the coffee breaks at the  conference) and for stimulating discussions,  and G. B\"ockle for helpful correspondence. We are especially grateful  to J-P. Serre for  a lively e-mail  correspondence  during the quarantine period in March--May 2020, which has contributed substantially to \S  2 and \S 3 of this paper.   We are grateful to Ali Cheraghi for his observation recorded in Lemma \ref{ali}  and Corollary \ref{classification} which greatly improved our previous results in \S \ref{Finite Groups}.

ML was  partially supported by the NSF grant  DMS-2001349.

\section{Liftable finite groups}\label{Finite Groups}

\subsection{Definition and preliminaries}
    \begin{defn}
 We say that a (topological) group $G$ is \emph{liftable}, or \emph{has property $L$},  or \emph{is $L$},  if  any  (continuous) homomorphism $f: G \rightarrow \GL_n(k)$,  with $n$ a positive integer and $k$ a finite field, lifts to a (continuous)  homomorphism  $g:G \rightarrow \GL_n(W_2(k))$.

We say that a (topological) group $G$ is $p$-liftable, or has property $L_p$, or is $L_p$,   if  any (continuous)  homomorphism $f: G \rightarrow \GL_n(k)$, with $n$ a positive integer and $k$ a finite field of characteristic $p$,  lifts to a (continuous)  homomorphism  $g:G \rightarrow \GL_n(W_2(k))$.
\end{defn}

\begin{lem}\label{reduction to p groups}
A finite group  $G$  is $L_p$ if and only if  a Sylow $p$-subgroup   of $G$  is $L_p$.
\end{lem}

\begin{proof}
 A   homomorphism $f: G \rightarrow \GL_n(k)$,  with $n$ a positive integer and $k$ a finite field, lifts to a homomorphism  $g:G \rightarrow \GL_n(W_2(k))$ if and only if the corresponding obstruction class $\alpha_f \in H^2(G,M_n(k))$ vanishes. If $P$ is a Sylow $p$-subgroup of $G$, the restriction of $\alpha_f$ to $P$ vanishes if and only if $\alpha_f=0$. This proves that if $P$ is $L_p$, then $G$ is $L_p$.
 
 To prove the converse, let  $g:P \rightarrow  \GL_m(k) $ be  a non-liftable  homomorphism afforded by a $P$-module
 $V_g$, a $k$-vector space of dimension $m$. Then we assert that the induction $V_f$ of $V_g$ from $P$ to $G$,  associated to a  homomorphism $f: G \rightarrow \Aut(V_f)$ is non-liftable.  To see this, observe that $V_f$ as a $P$-module has  $V_g$ as a direct summand, and the projection of $\alpha_f|_P$ to  the summand $H^2(P, \End(V_g))$ of  $H^2(P, \End(V_f))$ is  $\alpha_g$ and hence non-zero.

\end{proof}

 Ali Cheraghi pointed out to us the following strengthening of  one direction  of  Lemma \ref{reduction to p groups} upon using Mackey theory.

\begin{lem}\label{ali}
A finite group  $G$ is $L$ (or $L_p$) if and only every subgroup of $G$ is $L$ (or $L_p$). 
\label{semidirect}
\end{lem}

\begin{proof}
Only one direction needs proof.
The proof is close to the proof of Lemma \ref{reduction to p groups}. Let  $g:H \rightarrow  \GL_m(k) $ be  a non-liftable  homomorphism afforded by a $H$-module
 $V_g$, a $k$-vector space of dimension $m$. Then we assert that the induction $V_f$ of $V_g$ from $H$ to $G$,  associated to a  homomorphism $f: G \rightarrow \Aut(V_f)$ is non-liftable.  To see this, observe that $V_f$ as a $H$-module has  $V_g$ as a direct summand by Mackey theory, and thus $\End(V_g)$ is a direct summand of the $H$-module $\End(V_g)$ which implies that   $H^2(H, \End(V_g)) \hookrightarrow H^2(G, \End(V_f)) $.  Consider the obstructions $\alpha_f$ and $\alpha_g$ to lifting $V_f$ and $V_g$ to $W_2(k)$.  We deduce that  the projection of $\alpha_f|_H$ to  the summand $H^2(H, \End(V_g))$ of  $H^2(H, \End(V_f))$ is  $\alpha_g$ and hence $\alpha_f$ is non-zero. 
\end{proof}

\begin{prop}\label{implication}
Guess~\ref{local} implies Guess~\ref{strong}.
\end{prop}

\begin{proof}
It is clear that the Sylow subgroups of the three families of groups in  Guess~\ref{strong} are all in $\{\Z/3\Z\}\cup
 \{\Z/2^n\Z\mid n\in \N\}$.
We claim that these are the only groups for which this is true.  If $G$ is such a group, it must be of order $2^n$ or $3\cdot 2^n$.  In the first case we are done.  In the second
case, let $G_2\cong \Z/2^n\Z$ and $G_3\cong \Z/3\Z$ be Sylow subgroups, and let $y$ be a generator of $G_3$.  
The order of the automorphism group of any cyclic $2$-group is a power of $2$, so $y$ centralizes
any $2$-group in $G$ which it normalizes.
As $G_2$ is of index $3$ in $G$, either it is normal or it contains a subgroup $G'_2$ which is normal
in $G$ and for which $G/G'_2\cong S_3$.  In the former case, $y$ normalizes $G_2$, so $G\cong \Z/3\Z\times G_2$.  

We therefore assume that there exists $G'_2$ as above.  As $y$ centralizes $G'_2$, we have
$G_3 \times G'_2$ is an index $2$ subgroup of $G$.  Thus, the normalizer of $G_3$ in $G$ has index $\le 2$, which, by the third Sylow theorem, implies $G_3$ is normal in $G$.  The conjugation action of $G_2$ on $G_3$ factors through $G_2/G'_2\cong \Z/2\Z = \Aut(\Z/3\Z)$, so there is a unique non-direct semidirect product $\Z/3\Z \rtimes  \Z/2^n\Z$, and $G$ must be isomorphic to that product.

We conclude the proof by appealing to Proposition \ref{reduction to p groups}.
\end{proof}

\begin{rem}
If it turns out that generalized quaternionic groups  $Q_{2^n}$ (of order $2^n$)  are liftable, as their automorphism groups are again of $2$-power order, the proof above will show that the liftable groups are  $\Z/2^n\Z$, $\Z/3\Z\times\Z/2^n\Z$,  $\Z/3\Z\rtimes \Z/2^n\Z$, $Q_{2^n}$, $\Z/3\Z\times Q_{2^n}$, or $\Z/3\Z\rtimes Q_{2^n}$.
\end{rem}

\begin{quest}
Is it true that every   group $G$  that is $L$ has the property that any homomorphism $f:G \rightarrow \GL_n(k)$ lifts  in fact all the way to $\GL_n(W(k))$? Is this true when restricted to finite groups $G$?
\end{quest}

\subsection{Liftable cyclic groups}

\begin{prop}
\label{cyclic p-groups}
     A cyclic group of prime order $p$ has property $L$ if and only if $ p\leq 3$.
\end{prop}
    
 \begin{proof}  The  proof that follows  is due to Serre.
 
 \begin{enumerate}

 \item  $p=2$:
 Let $k$ be any field of characteristic 2.
 Then we claim that every element $ x$ of order $2$ of  $\GL_n(k)$ lifts to an element of
order 2 of $\GL_n(A)$, where  $A$  is any commutative ring with a surjective
homomorphism  $A \to k$. We only have to check this when $ k = \F_2$ and $x$ is a Jordan matrix of size $n = 2$, which we may assume is of the form    \[   \left( \begin{array}{cc} 1 & 1 \\  0 & 1 \end{array} \right).\] This lifts to
\[   X = \left( \begin{array}{cc} -1 & 1 \\ 0 & 1 \end{array} \right)\] in $\GL_2(A)$.
 
 \item   $p=3$:  Let $k$ be any field of characteristic 3.
 Then we claim that every element $ x$ of order $3$ of  $\GL_n(k)$ lifts to an element of
order 3 of $\GL_n(A)$,  where  $A$  is any commutative ring with a surjective
homomorphism  $A \to k$. We only have to check this when $ k = \F_3$ and $x$ is a Jordan matrix of size $n = 2$ or $ 3$.
  
  \begin{itemize}
  
  \item  $n = 2$:  Consider the algebra $B=A[X]/(X^2+X+1)$ which is free of rank $2$ over $A$,
and view  multiplication by $X$  an an $A$-linear endomorphism of $B \cong  A^2$. We have  $X^2+X+1= 0$, hence
$X^3=1$.  The matrix of  $X$ is of the form    \[   \left( \begin{array}{cc} 0 & 1 \\ -1 & 1 \end{array} \right).\] By reduction mod 3, we have an element of order 3
of $\GL_2(k)$, hence what we want.

 \item $n=3$:  Consider multiplication by $X$ as an endomorphism of $B= A[X]/(X^3-1)$ regarded as a free module of rank $3$ over $A$. One has to check that
its reduction mod 3 is a Jordan matrix of size 3. But this means looking at multiplication
by  $x \in  k[x]/(x^3-1)$, which is the same as  $k[t]/(t^3)$, and we are done.
  
  \end{itemize}
  
\item  $p>3$:  This is well known, but we reprove it as part of Proposition~\ref{odd power} below.

\end{enumerate}

\end{proof}

We add to the list of liftable groups.

\begin{prop}
\label{power of 2}
 The group $G=\Z/2^n\Z$ is liftable.
\end{prop}

\begin{proof}
For odd primes $p$, the $L_p$ condition is trivial.  For the $L_2$-condition, we prove more:
when $k$ is of characteristic $2$, every element $x$ of order $2^n$ in  $\GL_m(k)$ lifts to an element of order $2^n$ in  $\GL_m(A) $ where  $A$  is any commutative ring with a surjective
homomorphism  $A \to k$.  It suffices to prove this for $A=\Z$.  We  decompose $x$ into Jordan blocks, all of which must have eigenvalue $1$, and all of which have order dividing $2^n$.
By induction on $n$, we may assume that each block is of order $2^n$, and it suffices to treat the case of a single block, which means $2^{n-1}<m\le 2^n$.

The irreducible factors
$$u-1,u+1,u^2+1,\ldots,u^{2^{n-1}}+1$$
of $u^{2^n}-1$ have degrees $1,1,2,4,\ldots,2^{n-1}$, and $m$ can be written as a sum of terms in this sequence.  It follows that there exists a polynomial
$P(u) \in \Z[u]$ of degree $m$ which divides 
$u^{2^n}-1$.  

Consider the $\Z$-linear endomorphism given by multiplication by $X$ on the free $\Z$-module $M=\Z[X]/(P(X))$.  
Its mod 2 reduction is given by multiplication by $x$ on $\F_2[x]/(\bar P(x))$.  As $u^{2^k}+1\equiv (u-1)^{2^k}\pmod 2$,
it follows that $\bar P(x) = (x-1)^m$, so multiplication by $x$ has a single Jordan block with eigenvalue $1$.

\end{proof}

\begin{prop}
\label{odd power}
If $p\ge 3$ and $n\ge 2$ or $p\ge 5$ and $n\ge 1$, then $\Z/p^n\Z$ does not have property $L_p$.
\end{prop}

\begin{proof}
Let $m=p^{n-1}+1$, so a unipotent Jordan block $x$ of size $m$ over $\F_p$ has order $p^n$.
Note that our hypotheses mean that the order is at least $2m+1$.
We claim that no lift $X$ of $x$ to $\GL_m(\Z/p^2\Z)$ has order $p^n$.
Indeed, any such $X$ can be written $X = I + N + p M$, where $N$ is a nilpotent matrix, so $N^m=0$.
Therefore,
\begin{align*}
X^{p^n} &= I + \sum_{i=1}^p \binom{p^n}{ip^{n-1}} (N+pM)^{ip^{n-1}} \\
&= I + \sum_{i=1}^{p-1} \binom{p^n}{ip^{n-1}} N^{ip^{n-1}} + N^{p^n} + p\sum_{j=0}^{p^n-1}
N^j M N^{p^n-1-j}.
\end{align*}
As $N^j=0$ or $N^{p^n-1-j}=0$ for all $j$, we have
$$X^{p^n} = I + p\sum_{i=1}^{p-1}\binom{p^n}{ip^{n-1}} N^{ip^{n-1}}.$$
As $(x-1)^{p^{n-1}}$ does not lie in the span of the set of higher powers of $x-1$ in $M_n(\F_p)$, it follows that $X^{p^n}\neq I$.
\end{proof}

\subsection{Liftable finite abelian groups}

We quote a result from \cite{serre:illusie}.

\begin{prop}(Serre)\label{serre}
  Let $k$ be a field of characteristic $p$, and $A$ a local ring  with residue field $k$. Let $1 <n \leq p$, and let $G$ be an abelian  subgroup of $\GL_n(k)$  of type $(p,p)$ such that for each $g\in G\setminus\{e\}$,  $g-1$ has kernel of dimension $n-1$. If $G$ lifts to $\GL_n(A)$, then $pA=0$.
\end{prop}

\begin{prop}
\label{p times p}
 The group $\Z/p\Z \times \Z/p\Z$ does not have property $L_p$ for any prime $p$.  More precisely, there are   injective homomorphisms $f: \Z/p\Z \times \Z/p\Z \to \GL_2(\F_{p^2})$, and any such $f$ does not lift to $\GL_2(W_2(\F_{p^2}))$.
 
\end{prop}

\begin{proof}
We just have to observe that for all primes $p$,  there is an injective homomorphism  $f: \Z/p\Z \times \Z/p\Z \to \GL_2(\F_{p^2})$ whose image  satisfies the conditions of  the proposition above.
\end{proof}

Extending Serre's method  in  \cite{serre:illusie}, we get the following:

\begin{prop}
\label{two powers of 2}
For $m,n\ge 1$ integers, $\Z/2^m\Z\times \Z/2^n\Z$ does not have property $L$.
\end{prop}

\begin{proof}
Without loss of generality, we assume $m\ge n$.  
Let $\omega\in\F_4$
denote an element of order $3$.  Let $x\in M_{2^m}(\F_4)$ denote a Jordan block of size $2^m$ with eigenvalue $0$.  Let $y=\omega x^{2^{m-n}}$.
The matrices $I+x$ and $I+y$ are of order $2^m$ and $2^n$ respectively, and they commute.  
We prove that the representation $\rho\colon \Z/2^m\Z\times \Z/2^n\Z\to \GL_{2^m}(\F_4)$
defined by 
$$\rho(i,j)= (I+x)^i(I+y)^j$$ 
does not lift to $\GL_{2^m}(W_2(\F_4))$.  

Let $A = W_2(\F_4)$.
By induction on $k\ge 1$, for every commutative $A$-algebra $B$ and every finite sequence $b_1,\ldots,b_r\in B$,
\begin{equation}
\label{two powers}
(b_1+\cdots+b_r)^{2^k} = \sum_{i=1}^r b_i^{2^k} + 2\sum_{1\le i < j\le r} b_i^{2^{k-1}} b_j^{2^{k-1}}.
\end{equation}

Suppose $X,Y\in M_{2^m}(A)$ lie over $x$ and $y$ respectively so that $(I+X)^{2^m}=I$,
$(I+Y)^{2^n}=I$, and $X$ and $Y$ commute.  Let $B:=A[X]$ denote the $A$-subring of $M_{2^m}(A)$ generated by $X$.
We have $(I+X)^{2^m} = I$, so by \eqref{two powers}, 
\begin{equation}
\label{reduction}
X^{2^m} = -2X^{2^{m-1}} = 2X^{2^{m-1}}.
\end{equation}
Therefore, $X^{2^{m+1}} = 4X^{2^m} = 0$, so $X$ is a nilpotent element of $B$.
In particular, $B$ is a local ring with residue field $\F_4$.

As a $B$-module, the $M_{2^m}(A)$-module $A^{2^m}$ is free of rank $1$.
Indeed, as $X$ acts on $(A/2A)^{2^m} = \F_4^{2^m}$ as a nilpotent matrix with a single Jordan block, 
there exists  $v\in A^{2^m}$
which reduces to a cyclic vector in $\F_4^{2^m}$.
For any such $v$, Nakayama's lemma implies
 $\{v, Xv, \ldots, X^{2^m-1}v\}$
is a basis for $A^{2^m}$ as $A$-module, so $\{v\}$ is a basis for the $B$-module structure.  
It follows that the centralizer of $X$ in $M_{2^m}(A)$ consists of polynomials in $X$ with coefficients in $A$.  We can therefore write
$$Y = \sum_i a_i X^i .$$

By \eqref{two powers},
$$I = (I+Y)^{2^n} = I + \sum_i a_i^{2^n} X^{i2^n} + 2\sum_i a_i^{2^{n-1}} X^{i2^{n-1}} +  2\sum_{i<j}a_i^{2^{n-1}} a_j^{2^{n-1}} X^{(i+j)2^{n-1}}.$$
Let $k=m-n$.
The mod 2 reduction of $Y$ is $y=\omega x^{2^k}$, so  $a_i\in 2A$ for $i<2^k$, 
while $a_{2^k}$ reduces to $\omega$ mod $2$, so $a_{2^k}^{2^n}+a_{2^k}^{2^{n-1}}$ is a unit $u\in A$.  By \eqref{reduction},
\begin{align*}
0 &= \sum_{i\ge 2^k}  a_i^{2^n} X^{i2^n} + 2\sum_{i\ge 2^k} a_i^{2^{n-1}} X^{i2^{n-1}} +  2\sum_{2^k\le i<j}a_i^{2^{n-1}} a_j^{2^{n-1}} X^{(i+j)2^{n-1}} \\
&=\sum_{i\ge 2^k} a_i^{2^n} X^{i 2^n} + \sum_{i\ge 2^k} a_i^{2^{n-1}} X^{(i+2^k)2^{n-1}} \\
& \equiv \bigl(a_{2^k}^{2^n}+a_{2^k}^{2^{n-1}}\bigr)X^{2^m}\equiv uX^{2^m}\pmod{X^{2^m+1}}.
\end{align*}
 
As  $X$ is nilpotent, this implies $X^{2^m} = 0$.  By \eqref{reduction}, $X^{2^{m-1}} \in 2 M_{2^m}(A)$, which is impossible since its mod 2 reduction $x^{2^{m-1}}$ is non-zero.

\end{proof}

\begin{rem}
\label{nonrigid}
Let $R = W(\F_4)[\sqrt 2]$.  The representation $\rho\colon \Z/4\Z\times \Z/2\Z\to \GL_{4}(\F_4)$ of Proposition~\ref{two powers of 2} lifts to $R/(2\sqrt 2)$.
To see this, we use Proposition~\ref{power of 2} to find $X\in \GL_4(R)$ such that $(I+X)^4 = I$ and $X$ reduces to a single nilpotent Jordan block over $\F_4$.
Thus,
$$X^4 \equiv 2X^2 \pmod{2\sqrt 2}.$$
Setting
$$Y = \sqrt 2 X + \zeta_3 X^2,$$
clearly $I+X$ and $I+Y$ commute, and $Y$ reduces to $\omega x^2$ in $M_4(\F_4)$, where $\omega$ denotes the reduction of $\zeta_3$.  
Squaring $I+Y$, we obtain
\begin{align*}
(I+Y)^2 &= I + 2X^2 + \zeta_3^2 X^4 + 2\sqrt 2 X + 2\zeta_3 X^2 + 2\sqrt 2\zeta_3 X^3 \\
& \equiv I + 2X^2 + \zeta_3^2 X^4 + 2\zeta_3 X^2 \\
& \equiv I + 2(1+\zeta_2^2+\zeta_3)X^2 \equiv I\pmod {2\sqrt 2}.
\end{align*}
\end{rem}

Our results so far lead to a classification of liftable finite abelian groups.

\begin{prop}\label{abelian}
A finite abelian group is $L$ if and only if it is cyclic of order $2^n$ or order $3\cdot 2^n$
for some non-negative integer $n$.
\end{prop}

\begin{proof}
It suffices to prove that the  abelian $p$-groups which are $L$ are exactly the cyclic groups
of order $2^n$ and $3$.  In the positive direction, this follows from Propositions~\ref{cyclic p-groups} and \ref{power of 2}.  In the negative direction, it follows from 
the direct product case of Lemma~\ref{semidirect}, together with the basic cases 
in Propositions \ref{cyclic p-groups}, \ref{odd power},  \ref{p times p},  and \ref{two powers of 2}.

\end{proof}

We owe the following corollary to A. Cheraghi.

\begin{cor}\label{classification}
 For $p>3$ the only finite $p$-group that is $L_p$ is the trivial group. The only  finite 3-groups that are $L_3$ are the trivial group and $\Z/3\Z$.  For $p=2$, $\Z/2^n\Z$ is $L_2$, and if a finite 2-group $G$ is $L_2$, then $G$ is either cyclic or quaternionic.
\end{cor}

\begin{proof}
 This follows from  Lemma \ref{ali}, Proposition \ref{abelian} and the classification of finite  $p$-groups with the property that every abelian subgroup is cyclic.
\end{proof}

\begin{rem}
 The smallest $2$-group for which we do not know whether $L_2$ holds is the quaternionic group $Q_8$ of order 8. \end{rem}

\subsection{Strong rigidity for finite groups}\label{Rigidity}

As Serre suggested to us, one could look at a stronger rigidity  property of  a finite  subgroup $G$ of $\GL_n(k)$ ($\chr(k) = p$),  than just not lifting to $\GL_n(W_2(k))$.

 We say that a map $G \to \GL_n(k)$ is {\it rigid}
if  for   any  local Artinian  ring $A$ with residue field $k$ such that $G$ lifts to $\GL_n(A)$,  we have that $pA=0$. Proposition \ref{serre} gives examples of rigid embeddings of $G=\Z/p\Z \times \Z/p\Z$ in $\GL_n(k)$.
However, Remark~\ref{nonrigid} shows that the embedding of $\Z/4\Z \times \Z/2\Z$ given in Proposition~\ref{two powers of 2} is not rigid.
If $pA=0$, then $k \hookrightarrow A$, and we call the resulting lift of $G$ to $\GL_n(A)$ the trivial, or constant,  lift.    If $G \to  \GL_n(k) $  has the further property that if  it  lifts to $\GL_n(A)$, then $pA=0$ and the lift  is conjugate  to  the trivial   lift,  then we say that   $G \to  \GL_n(k) $  is {\it strongly rigid}.  Note that rigidity  of $ G \to \GL_n(k)$ is a {\it relative} property of $G$ with respect to a homomorphism $G \to \GL_n(k)$,  while
the liftable property $L$ depends only on the group $G$.

 Let $\Ad$ be the $G$-module $M_n(k)$ under conjugation action of $G$.

\begin{prop}\label{rigidity} $ $
 The following are equivalent:

 \begin{itemize}
 
 \item  $ G \to \GL_n(k)$ satisfies the two properties:
\begin{enumerate}
 
 \item   $G \to \GL_n(k)$ does not lift to $\GL_n(W_2(k))$;
 
 \item $H^1(G,\Ad)=0$.
 \end{enumerate}
 
 \item   The embedding $G \hookrightarrow \GL_n(k)$  is strongly  rigid.

 \end{itemize}
 
 \end{prop}

\begin{proof}
 Assume $ G \to  \GL_n(k)$ satisfies  properties (1)  and  (2), and can  can be lifted to  $\GL_n(A)$ with $(A,\m)$  a local Artinian   ring with residue field $k$.   We want to show that $pA=0$. We may assume $p\m = 0$. We have a lifting of $G$ to $\GL_n( A)$. This gives a lifting  of $G$ to $\GL_n( A/pA)$. On the other hand, we have the trivial (constant) lifting of $G$ to $\GL_n(A/pA)$ since $A/pA$ contains $k$. Since 
$$H^1(G,M_n(\m/pA))\cong H^1(G,M_n(k))^{\dim_k \m/pA} = 0,$$
the two liftings are conjugate, so we may identify them.  

As $p^2=0$ in $A$, the universal property for Witt rings gives a homomorphism $i\colon W_2(k)\to A$.  
If $\alpha_{W_2(k)} \in H^2(G,M_n(pW_2(k))$ and $\alpha_A \in H^2(G,M_n(pA))$ denote, respectively, the obstruction to lifting
$G \to \GL_n(k)$ to $\GL_n(W_2(k))$ and $G \to  \GL_n( A/pA)$  to $\GL_n(A)$,
then $\alpha_{W_2(k)}\mapsto \alpha_A$ under the map 
$$H^2(G,\Ad) = H^2(G,M_n(pW_2(k))\to  H^2(G,M_n(pA)) = H^2(G,\Ad)\otimes_k (pA).$$
This map is given by the identity on $H^2(G,\Ad)$ tensored with the restriction of $i$ to
$pW_2(k)\cong k$.  Since $\alpha_{W_2(k)} \neq 0$ and $\alpha_A = 0$, it follows that $i(pW_2(k)) = 0$, so $i(p)=p=0$ in $A$, so $pA = 0$.
We deduce that $G \to  \GL_n(k)$  is  strongly  rigid. The converse direction is easy. 
\end{proof}

\begin{rem} $ $

\begin{itemize}

\item   If $G \subset \GL_n(k)$ acts  absolutely irreducibly on $k^n$, the deformations of $G$ to local $W(k) $-algebras  $A$ with residue field $k$ is representable by  a local  $W(k)$-algebra $R$ with residue field $k$. The fact that $H^1(G,\Ad)=0$, implies that the Zariski mod $p$ tangent space $\m/(\m^2,p)=0$,  with $\m$ the maximal ideal of $R$, and thus that  $R$ is a quotient of $W(k)$. If $G$ does not lift to $W_2(k)$ we further get that  $R=k$.

\item The rigid example of $G=\Z/p\Z \times \Z/p\Z \hookrightarrow \GL_n(k)$ of Proposition \ref{serre} is not  strongly rigid, in fact $H^1(G,\Ad) \neq 0$ as we prove below. 
\end{itemize}

\end{rem}

\begin{lem}
 If  $k$ is a field of characteristic  $p$, $G$ a finite $p$-group that is not cyclic,  and $V$ is a finite-dimensional $k[G]$-module  on which some element  $g \in G$ acts on $V$  as a single Jordan block, then $H^1(G,V)\neq 0$. 
 
 \end{lem}
 
\begin{proof}
Let $ h^i(G,V)={\rm dim}_k H^i(G,V)$ for a $k[G]$-module $V$. The case of $V$ of dimension 1 over $k$ is easy, so we assume that ${\rm dim}_k(V)>1$.
Under our hypothesis for any non-zero  $k[G]$-submodule or $k[G]$-quotient module $W$ of $V$, $h^0(G,W)=1$. If we let  $V_1  = V^G$, then $V_1$ is one-dimensional.  Let $V_2 = V/V_1$, and use the cohomology sequence for $0\to V_1\to V\to V_2\to 0$.  We get:
$$0 \to k \to k \to k \to H^1(G,k) \to H^1(G,V).$$   As $G$ is not cyclic $h^1(G,k) \geq 2$, so $h^1(G,V)\ge 1$.
\end{proof}

The following corollary responds to a question Serre asked us: when  does $\SL_n(k) \hookrightarrow \GL_n(k)$ have the
property that if there is a  lift of $\SL_n(k)$ to $\GL_n(A)$ for  a local ring $A$ with residue field $k$ then $pA=0$?

\begin{cor}\label{rigid finite Lie groups}
For $n>1$, and for finite fields $k$ of characteristic $p$  with $(p, n)=1$, $\SL_n(k) \hookrightarrow \GL_n(k)$ is strongly rigid except when:

(i) $n=2$ and $|k| \le 5$;

(ii) $n=3$ and $|k| \leq 3$.
\end{cor}

\begin{proof}
We need to verify when  $\SL_n(k) \hookrightarrow \GL_n(k)$ satisfies properties 1 and 2 of Proposition \ref{rigidity}.

An exercise in \cite[IV-27]{serre:ladic} asserts that for $p \ge 5$ and $n\ge 2$, the only closed subgroup of $\SL_n(\Z_p)$
mapping onto $\SL_n(\F_p)$ is the full group.  This implies that $\SL_n(\Z/p^2\Z)\to \SL_n(\F_p)$ does not split.  As
$\SL_n(\F_p)$ is perfect, the composition of any homomorphism $\SL_n(\F_p)\to \GL_n(\Z/p^2\Z)$ with the determinant map
must be trivial, and it follows that the embedding $\SL_n(\F_p)\to \GL_n(\F_p)$ does not lift to $\SL_n(\F_p)\to \GL_n(\Z/p^2\Z)$.
In fact, for $n\ge 2$, it is known \cite[Theorem~1.3]{vasiu-surj} that $\SL_n(k)$ does not lift to $\SL_n(W_2(k))$ (or, therefore to $\GL_n(W_2(k))$) except when $n=2, |k|=2,3$ and $n=3,|k|=2$.

Regarding property 2, 
assuming $p$ does not divide $n$, $\Ad$ decomposes as a direct sum $k\oplus (\Ad/k)$.  
As $\SL_n(k)$ is perfect, $H^1(\SL_n(k),k) = 0$.
A theorem of Cline, Parshall, and Scott \cite[Table (4.5)]{cline-parshall-scott} asserts that if $p$ does not divide $n$,
then $H^1(\SL_n(k),\Ad/k)=0$ if $|k|\ge 7$ or $n\ge 3$ and $|k|\ge 4$.  The cases $|k|=2$ and $|k|=3$ are covered by \cite[Tables B and C]{jones-parshall}.

\end{proof}

\section{Two dimensional representations and negligible classes}\label{Negligible}

Lur'e \cite{lurye} in 1990 introduced the following  property of cohomology classes in  $H^i(G,M)$ for finite groups  $G$  and finite $G$-modules $M$ that  was further 
 studied by Serre    in  his Coll\`ege de France course 1990--91   \cite[pp.~212--222, pp.~435--442]{serre:oeuvres}.
\begin{defn}
 Given a finite group $G$ and  a  finite $G$-module $M$, we say that a class $\alpha \in H^i(G,M)$ is \emph{negligible} if   for any  field $L$ with absolute Galois group $G_L$, and   any  continuous map $f: G_L \to G$,  $f^*(\alpha)=0$.
\end{defn}


We begin by remarking that the lifting question Question \ref{galois}  for   representations $\rhobar:G_L \to \GL_n(k)$, when $k$ is a finite field of characteristic $p$ and $L$ has characteristic $p$, is easily answered in the positive by the following lemma.

\begin{lem}\label{trivial}
 If $L$ is a field of characteristic $p$, and $k$ is a  finite field of characteristic $p$, then for any $k[G_L]$-module $M$, with $M$ a $k$-vector space of finite dimension, $H^2(L,M)=0$.
\end{lem}

\begin{proof}
 This is a standard consequence of Artin-Schreier theory, see \cite[Theorem 1.7]{bockle:modp2}.
\end{proof}

The following result is proved using the method of \cite[Theorem 1]{K-JNT} (see also \cite[Theorem 6.1]{mfcdc}).

\begin{prop}\label{lifting n=2}

If $L$ is a field and $G_L$ its absolute Galois group,  for any  field $k$ of characteristic $p$, and any ring $R$ with a surjective map $R \to k$ and $p^2R=0$,  a continuous homomorphism 
    $f:G_L \rightarrow \GL_2(k)$ lifts to  a  continuous homomorphism  $g: G_L \rightarrow \GL_2(R)$.
\end{prop}

    \begin{proof}     
     We may restrict  by Lemma  \ref{trivial} to  the case  where $L$ has characteristic different from $p=\chr(k)$. The obstruction to lifting is a class
     $\alpha \in H^2(G_L,\Ad$, with $G_L$ acting upon $\Ad=M_2(k)$ by composing $f$ with the adjoint representation of $\GL_2(k)$.  As the restriction map $ H^2(G_L,\Ad) \to H^2(G_M,\Ad)$  is injective, for $M$ a finite extension of $L$ of degree prime to $p$,   we may assume:
     
     \begin{itemize}
     
     \item $L$ contains a  primitive $p$th root of unity, as $(p,[L(\mu_p):L])=1$; 
     \item  the image of $f$ is a  non-trivial  $p$-subgroup of $\GL_2(k)$, of type $(p,\ldots,p)$ as the Sylow $p$-subgroup of $\GL_2(k)$ is of type $(p,\ldots,p)$, and contained in the unipotent matrices  
     \[\Bigl\{   \left( \begin{array}{cc} 1 & \ast \\ 0 & 1 \end{array} \right)\Bigm|* \in k\Bigr\}\] of $\GL_2(k)$. 
     \end{itemize}
     
     By Kummer theory the splitting field of $f$ is given by $L_1=L(x_i^{1 \over p})$ for some $x_i$, $i=1,\ldots, r$ that are independent  in $L^\times/{(L^\times)}^p$.  We construct the extension $L_2=L(\mu_{p^2}, x_i^{1 \over p^2})$ and a homomorphism 
     $$g: \gal(L(\mu_{p^2}, x_i^{1 \over p^2} )/L) \to  N= \Bigl\{ \left( \begin{array}{cc} \alpha& \beta \\ 0 & 1 \end{array} \right)\Bigm|\alpha \in  \langle1+p \rangle, \beta \in R\Bigr\} \leq \GL_2(R),$$  such
     that  for $\sigma \in  \gal(L(\mu_{p^2}, x_i^{1 \over p^2} )/L)$,  $g(\sigma)=  \left( \begin{array}{cc} \epsilon(\sigma) & \ast \\ 0 & 1 \end{array} \right),$ with $\epsilon:G_L \to {(\Z/p^2\Z)}^*$   the 
     mod $p^2$ cyclotomic character, that lifts $f$.  
     
     For this note that $\gal(L(\mu_{p^2}, x_i^{1 \over p^2} )/L)$ is isomorphic to  a subgroup of either the semidirect product of $ (\Z/p^2\Z)^r$ by the order $p$ subgroup of  ${(\Z/p^2\Z)}^*$ with the  diagonal action   if $\mu_{p^2} \not \subset L$,  or $ (\Z/p^2\Z)^r$ otherwise.
     This allows us to construct $g: G_L \to N$ which reduces mod $p$  to $f$.

            \end{proof}
    
        \begin{rem} $ $
        
        \begin{itemize}


       \item  It's not clear how to  construct $g$ given $f$. When $f$ is surjective, the d\'evissage used in the proof constructs explicitly a mod $p^2$  reducible lift  $g'$, coming from Kummer theory,  of the restriction of $f$ to  the subgroup  $G_K=f^{-1}(P)$ of $G_L$ where  $P$  is a Sylow  $p$-subgroup of ${\rm im}(f)$.  The restriction of the lift  $g$ to $G_K$  is typically not $g'$: for instance  this is not the case when   when $k=\F_p$ and $p>3$ and the image of $f$ contains $\SL_2(\Z/p\Z)$, as  the image of $g$  then contains $\SL_2(\Z/p^2\Z)$.

       \item  In the case when $L$ is a global field, any  lift to $\GL_2(W_2(k))$ might be forced to be ramified at more places than the residual representation. Thus  it seems unlikely that Galois groups of maximal extensions of $L$ unramified outside a fixed finite set of places are liftable.
 \end{itemize}        
        \end{rem}

        \begin{cor} Let $k$ be a finite field.
 The exact sequence $$0 \rightarrow M_2(k) \rightarrow \GL_2(W_2(k)) \rightarrow  \GL_2(k)  \rightarrow 0$$ gives rise to a class $\alpha_{\GL_2(k)} \in H^2(\GL_2(k),M_2(k))$, with the action of $\GL_2(k) $ on $M_2(k)$ given by conjugation. Furthermore  $\alpha_{\GL_2(k)} \neq 0$ for $|k|>3$, but $f^*(\alpha_{\GL_2(k)})=0$  in $H^2(G_L,M_2(k))$ for any field $L$ and any homomorphism   $f:G_L \rightarrow \GL_2(k)$.  Thus $\alpha_{\GL_2(k)}$ is {\it negligible}.

\end{cor}

This naturally leads to the question which is a reformulation of Question \ref{galois}:

\begin{quest}\label{negligible}
 Let $k$ be a finite field.
 The exact sequence $$0 \rightarrow M_n(k) \rightarrow \GL_n(W_2(k)) \rightarrow  \GL_n(k)  \rightarrow 0$$ gives rise to a class $\alpha_{\GL_n(k)} \in H^2(\GL_n(k),M_n(k))$, with the action of $\GL_n(k) $ on $M_n(k)$ given by conjugation. Is  $\alpha_{\GL_n(k)}$  {\it negligible} in Galois cohomology for all $n \geq 1$?
\end{quest}

We may even ask if every representation $f:G_L \to \GL_n(k)$ lifts to $\GL_n(W(k))$. The $n=1$ case is trivial; we do  not know the answer  for  $n\geq 2$. We can produce characteristic 0 liftings
 for  2-dimensional reducible  mod $ p$ representation by the same method as in the proof of Proposition  \ref{lifting n=2}.
         
        \begin{prop}\label{reducible n=2}
        If $L$ is a field and $G_L$ its absolute Galois group,  for any  field $k$ of characteristic $p$, and any  $W(k)$-algebra $R$ with a surjective map $R \to k$,  a continuous reducible representation 
    $f:G_L \rightarrow \GL_2(k)$ lifts to  a  continuous representation $g: G_L \rightarrow \GL_2(R)$.
\end{prop}

\begin{proof}
The basic principle that the proof relies on is the surjectivity of the maps
$H^1(L,\mu_{p^n}) \to H^1(L,\mu_{p^m})$ for all positive integers $1 \leq m \leq n$, which follows from Kummer theory. We  refer to          \cite[Theorem 2, pg. 393]{K-JNT} and its proof for further details.

\end{proof}

\section{Cup products over local and global fields}\label{cupping}

We address Question \ref{quest:cup-product} in the number theoretic case of number fields and local non-archimedean fields which are their completions. Proposition \ref{solve local} and Theorem \ref{key} answer it in the case of local and global fields. For us  this   is an essential ingredient for the results in \S \ref{Heisenberg}  about  lifting mod $p$ Heisenberg  representations to  twisted mod $p^2$ Heisenberg representations over number fields and non-archimedean local fields of characteristic zero. 

\subsection{General fields}

The exact sequence
$$ 0 \to \mu_p^{\otimes n} \stackrel{i}{\to}  \mu_{p^2}^{\otimes n} \stackrel{\pi}{\to}  \mu_p^{\otimes n} \to 0,$$
gives rise to the  (twisted) Bockstein long exact sequence:
\begin{equation}
\label{Bockstein}
\ldots \to H^k(L,\mu_p^{\otimes n} ) \stackrel{i_{k,n}}{\longrightarrow}  H^k(L, \mu_{p^2}^{\otimes n})   \stackrel{\pi_{k,n}}{\longrightarrow} H^k(L,\mu_p^{\otimes n}) \stackrel{d_{k,n}}{\longrightarrow} \dots,
\end{equation}
where  the subscripts of $i$, $\pi$, and $d$ will usually be suppressed when they are clear from context.  We denote by $\delta\colon L^\times \to H^1(L,\mu_p)$ the composition of the quotient map
$L^\times \to L^\times/{L^\times}^p$ and the Kummer isomorphism $L^\times/{L^\times}^p\to H^1(L,\mu_p)$.

Let $p$ be a fixed prime.  

\begin{lem}
\label{kummer surjection}
If $L$ is any field, the map $\pi\colon H^1(L,\mu_{p^2})\to H^1(L,\mu_p)$ is surjective, and its image is $pH^1(L,\mu_{p^2})$.
\end{lem}

\begin{proof}
Via Kummer theory, $i_{1,1}$ and $\pi_{1,1}$ can be identified with the $p$th power map $L^\times/{L^\times}^p\to L^\times/{L^\times}^{p^2}$ and the reduction map $L^\times/{L^\times}^{p^2} \to L^\times/{L^\times}^p$.
The latter is obviously surjective with kernel ${L^\times}^p/{L^\times}^{p^2} = pH^1(L,\mu_{p^2})$.
\end{proof}

We say that a field $L$ \emph{has Property D} if for all $y_1,y_2\in H^1(L,\mu_{p^2})$ such that $\pi(y_1)\cup \pi(y_2)=0$
but $\pi(y_1)$ and $\pi(y_2)$ are not both zero, there exist $z_1,z_2\in H^1(L,\mu_p)$ such that
\begin{equation}
\label{D}
i(\pi(y_1)\cup z_1+\pi(y_2\cup z_2)) = y_1\cup y_2.
\end{equation}

\begin{lem}
\label{deform}
If $L$ has Property D and $x_1,x_2\in H^1(L,\mu_p)$ satisfy $x_1\cup x_2=0$, then there exist $\tilde x_1,\tilde x_2\in H^1(L,\mu_{p^2})$ such that $\pi(\tilde x_i) = x_i$ for $i=1,2$ and
$$\tilde x_1\cup \tilde x_2=0.$$
\end{lem}

\begin{proof}
If $x_1=x_2=0$, we are done, so we assume at least one $x_i$ is non-zero.
By Lemma~\ref{kummer surjection}, there exist $y_1,y_2\in H^1(L,\mu_{p^2})$ mapping to $x_1$ and $x_2$ respectively.
Thus 
$$\pi(y_1\cup y_2) = \pi(y_1)\cup \pi(y_2) = x_1\cup x_2 = 0.$$
By Property D, there exist classes $z_1,z_2$ satisfying \eqref{D}.  For all $z\in H^1(L,\mu_p)$ there exists $w\in H^1(L,\mu_{p^2})$ such that $\pi(w) = z$, and it follows that for all $y\in H^1(L\mu_{p^2})$,
\begin{equation}
\label{projection}
y\cup i(z) = y\cup i(\pi(w)) = y\cup pw = p(y \cup w) = i(\pi(y\cup w)) = i(\pi(y)\cup z).
\end{equation}
By Lemma~\ref{kummer surjection}, 
\begin{equation}
\label{p squared}
i(z_1)\cup i(z_2) \in p^2 H^2(L,\mu_p^{\otimes 2}) = 0.
\end{equation}
Therefore,
defining $\tilde x_1 := y_1 + i(z_2)$ and $\tilde x_2:= y_2 - i(z_1)$, we have $\pi(\tilde x_i) = \pi(y_i) = x_i$ and
\begin{align*}
\tilde x_1 \cup \tilde x_2 &= y_1\cup y_2 + i(z_2)\cup y_2 - y_1\cup i(z_1) - i(z_2)\cup i(z_1) \\
                                       &= y_1\cup y_2  - y_2 \cup i(z_2) - y_1 \cup i(z_1) \\
                                       &= y_1\cup y_2  - i(\pi(y_1)\cup z_1 + \pi(y_2)\cup z_2) = 0
\end{align*}
by \eqref{p squared}, \eqref{projection}, and \eqref{D}.
\end{proof}

We are grateful to G.~B\"ockle for pointing out part (4) in the following lemma to us.

\begin{lem}\label{div}
 
Let $L$  be a number field or a non-archimedean local field of characteristic zero, and let $p>2$ be prime. Then:

\begin{enumerate}

 \item  We have an exact sequence  $  H^2(L,\mu_p^{\otimes 2}) \to H^2(L,\mu_{p^2}^{\otimes 2}) \to H^2(L,\mu_p^{\otimes 2}) \to 0$;
in particular $\pi\colon H^2(L,\mu_{p^2}^{\otimes 2} ) \to H^2(L,\mu_p^{\otimes 2})$   is surjective.  If $\mu_{p^2} \subset L$
 then we have an exact sequence  $0 \to   H^2(L,\mu_p^{\otimes 2}) \to H^2(L,\mu_{p^2}^{\otimes 2} ) \to H^2(L,\mu_p^{\otimes 2}) \to 0$.

 \item The  kernel of the $\pi$-map  $H^2(L,\mu_{p^2}^{\otimes 2} ) \to H^2(L,\mu_p^{\otimes 2})$ is $p H^2(L,\mu_{p^2}^{\otimes 2})$.

 \item  When $L$ is local and $\mu_p \subset L$,  the cup product map   \[\cup:  H^1(L,\mu_p) \times H^1(L,\mu_p) \to  H^2(L,\mu_p^{\otimes 2 }  )\] 
 is identified with Tate's perfect duality pairing
 \[H^1(L,\mu_p) \times H^1(L,\mu_p) \to H^2(L,\mu_p)=\Z/p\Z.\]

\item  When $L$ is local and $\mu_p\subset L$, then  $i\colon H^2(L,\mu_p^{\otimes 2}) \to H^2(L,\mu_{p^2}^{\otimes 2})$ is injective if and only if $\mu_{p^2}\subset L$; otherwise it is the zero map.

\end{enumerate}

\end{lem}

 \begin{proof}
     
The first  sentence of part (1) follows from \eqref{Bockstein}
and the fact that for the fields $L$ we consider the $p$-cohomological dimension  of $L$ is $\leq 2$. When $\mu_{p^2}\subset L$, we can identify $\mu_{p^2}$ and $\mu_{p^2}^{\otimes 2}$,
and the second sentence follows from
Lemma~\ref{kummer surjection}.
     
For part (2), we have the obvious  inclusion $pH^2(L,\mu_{p^2}^{\otimes 2} ) \subset \ker\pi$. As $i\pi$ is multiplication by $p$, the   surjectivity of $\pi\colon H^2(L,\mu_{p^2}^{\otimes 2} ) \to H^2(L,\mu_p^{\otimes 2})$
from part (1) implies that this inclusion is an equality.

Part (3) is standard.

For part (4), fixing a trivialization of $\mu_p$ gives an isomorphism from $H^2(L,\mu_p^{\otimes 2})$ to $\Br_p(L)$, which is canonically identified with $\Z/p\Z$,
so the map $i\colon H^2(L,\mu_p^{\otimes 2}) \to H^2(L,\mu_{p^2}^{\otimes 2})$
is either injective or zero.  
Injectivity  is equivalent to
the surjectivity of the dual map, which by local duality is
$$H^0(L,\mu_{p^2}^{\otimes -1}) \to H^0(L,\mu_p^{\otimes -1}).$$
As $\mu_p\subset L$, the right hand term is identified with $\Z/p\Z$ by our choice of generator of $\mu_p$.  The left hand term is 
identified with $\Z/p^2\Z$ or $p\Z/p^2\Z$ depending on whether $\mu_{p^2}$ is or is not contained in $L$,
which gives the claim.

    \end{proof}

\subsection{Local fields}

\begin{prop}\label{solve local}
Let $L$  be a non-archimedean local field of characteristic zero and  $p>2$ a prime such that  $\mu_p \subset L$.  Then  any  $x_1,x_2 \in H^1(L,\mu_p)$ with  $x_1 \cup x_2=0$ lift to elements $\tilde x_i \in H^1(L,\mu_{p^2})$ such that $\tilde x_1 \cup  \tilde x_2=0$.
  
\end{prop}

\begin{proof}
By Lemma~\ref{deform}, it suffices to prove that $L$ has Property D.  Let $x_i = \pi(y_2)$.
Without loss of generality, we may assume $x_1\neq 0$.  We have $y_1\cup y_2 \in \ker \pi$,
so there exists $t\in H^2(L,\mu_p^{\otimes 2})$ such that $y_1\cup y_2 = i(t)$.  By Lemma~\ref{div} (3), there exists $z_1$ such that $x_1\cup z_1 = t$, and setting $z_2 = 0$,
equation \eqref{D} holds.
\end{proof}

\subsection{Number fields}

We start with a couple of preliminary results. The following lemma is used in Proposition \ref{fix t} and Theorem \ref{key}.

\begin{lem}
\label{local d}
Let $p>2$ be prime, and let $L$ be any local field of residue characteristic $\ell\neq p$ which contains $\mu_p$ but not $\mu_{p^2}$.
Then $d:=d_{1,2}\colon H^1(L,\mu_p^{\otimes 2})\to H^2(L,\mu_p^{\otimes 2})$ is surjective, and its kernel consists of all classes of the form $\delta a$, where 
$a$ is any element of $L$ of valuation $0$.
\end{lem}

\begin{proof}
By Lemma~\ref{div} (4), $i=0$, so $d$ is surjective.  By Kummer theory, $H^1(L,\mu_p) \cong L^\times \otimes \F_p$.
We can write $L^\times$ as a product of $\Z$, the multiplicative group of the residue field $k$ of $L$, and a pro-$\ell$ group.
By hypothesis, $k^\times$ contains any element of order $p$, but by Hensel's lemma and the hypothesis $\mu_{p^2}\not\subset L$, it cannot contain an elements of order $p^2$.  Thus, $k^\times\otimes \F_p\cong \F_p$, and $L^\times\otimes \F_p\cong \F_p^2$.
As $H^2(L,\mu_p) \cong \Br_p(L)\cong \F_p$, the kernel of $i$ is $1$-dimensional.

The multiplicative group of valuation $0$ elements of $L$ maps to a $1$-dimensional subspace of $L^\times\otimes \F_p$.  The image of this subspace consists of classes represented by $1$-cocycles which factor through $\hat \Z$ and therefore lift to
cocycles with values in $\mu_{p^2}^{\otimes 2}$.  These classes therefore lie in $\ker d$.

\end{proof}

 \begin{prop}
 \label{fix t}
Let $L$ be a number field,  $p>2$ a prime, and  $\mu_p \subset L$.   Let  $x_1,x_2 \in H^1(L,\mu_p)$.
Suppose that $t\in H^2(L,\mu_p^{\otimes 2})$ has the property that $(x_1)_v=(x_2)_v=0$ implies $i(t)_v=0$.
Then there exists
 $t'\in H^2(L,\mu_p^{\otimes 2})$ such that $i(t') = i(t)$, and for all $v$ such that $(x_1)_v = (x_2)_v = 0$ 
 and $\mu_{p^2}\not\subset L_v$, we have $t'_v = 0$.
 \end{prop}
 
 \begin{proof}
 Let $\{u_1,\ldots,u_k\}$ denote the valuations of $L$ with $u_j(p)>0$ and $\mu_{p^2}\not\subset L_{u_j}$.
 By Lemma~~\ref{div} (4),  $d=d_{1,2}\colon H^1(L_{u_j},\mu_p) \to H^2(L_{u_j},\mu_p)$ is surjective.  By weak approximation, there exists $b\in L^\times$ such that $(d\delta b)_{u_j} = -t_{u_j}$ 
 for $j=1,\ldots,k$.  Replacing $t$ by $t + d\delta b$, we have $t_{u_j}=0$ and $i(t)$ is unchanged.
 We may therefore assume that $t_{u_j}=0$.
 
Let $\{v_1,\ldots,v_l\}$ denote the support of $t$.  
If there does not exist $j$ such that $(x_1)_{v_j} = (x_2)_{v_j}=0$,
then we are done, so without loss of generality, we assume that $j=1$ has this property.
We claim that there exist $w_1,\ldots,w_m$ with the following properties:

\begin{enumerate}
\item For each $j\in [1,m]$, either $(x_1)_{w_j}$ or $(x_2)_{w_j}$ is non-zero
\item There exists an element $b\in L^\times$ such that $d\delta b$ is supported exactly on $\{v_1,w_1,\ldots,w_m\}$.
\item The element $b$ in (2) can be taken to be congruent to $1$ modulo a high enough power of $p$ to guarantee that
$(\delta b)_w = 0$ for all $w$ such that $w(p)>0$.
\end{enumerate}

Assuming this is true, we can replace $t$ by $t+d\delta b^r$ for some $r\in \{1,2,\ldots,p-1\}$ such that
$t$ is supported on $\{v_2,\ldots,v_l,w_1,\ldots,w_m\}$. This gives an element of $H^2(L,\mu_p^{\otimes 2})$
whose support has a smaller set of valuations for which $x_1$ and $x_2$
both have zero component than $t$ does.  The proposition then follows by induction on the number of such valuations.

To prove the claim, let $\cO$ denote the ring of integers in $L$ and $\primep$ denote the prime ideal corresponding to $v_1$.
We are looking for an element $b\in\cO$ congruent to $1$ modulo a suitable power $p^k$, with the property that
$(b) = \primep \prod_j \primeq_j^{c_j}$, where each valuation $w_j$ corresponding to a $\primeq_j$ satisfies $w_j(p)=0$
and $w_j\not\in\{v_1,\ldots,v_l\}$; moreover, either $(x_1)_{w_j} \neq 0$ or $(x_2)_{w_j}\neq 0$ for each $w_j$.
This can be done provided that the prime ideals corresponding to valuations satisfying these conditions generate the
ray class group $H_{p^k}$ of fractional ideals prime to $p$ modulo principal fractional ideals with generators which are $1$ mod $p^k$.
However, if $x_1\neq 0$ then the Chebotarev density of primes $\primeq$ such that $x_1^{1/p}\in L_{\primeq}$
is $1-\frac 1p\ge \frac 23$, so the set of images of the corresponding Frobenius elements in $H_{p^k}$ contains at
least $\frac{2|H_{p^k}|}{3}$ elements and therefore generates $H_{p^k}$.
 
 \end{proof}

\begin{thm}\label{key}
Let $L$ be a number  field,  $p>2$ a prime, and  $\mu_p \subset L$.   Let  $x_1,x_2 \in H^1(L,\mu_p)$ satisfy  $x_1 \cup x_2=0$.
Then there exist elements $\tilde x_1,\tilde x_2\in H^1(L,\mu_{p^2})$ mapping by $\pi$ to $x_1$ and $x_2$ respectively such      
that $\tilde x_1\cup \tilde x_2 = 0$.
\end{thm}

\begin{proof} 
It suffices to prove that $L$ has Property D.

We fix once and for all a generator of $\mu_p$ and use it to identify $H^2(L,\mu_p^{\otimes 2})$ with $H^2(L,\mu_p) = \Br_p(L)$.
We set $x_j := \pi(y_j)$ and choose $a_j\in L^\times$ so that $\delta a_j = x_j$.
We choose $t\in H^2(L,\mu_p^{\otimes 2})$ such that $y_1\cup y_2 = i(t)$.
Observe that  $(x_1)_v=(x_2)_v=0$ implies by \eqref{p squared}
$i(t)_v = (y_1)_v \cup (y_2)_v=0$.
By Proposition~\ref{fix t}, we may assume without loss of generality that $(x_1)_v=(x_2)_v=0$ implies $t_v=0$.

We write $(a,b)$ for $\delta a \cup \delta b$.
Our goal is to represent $t$ as $(a_1,b_1)+(a_2,b_2)$ for some $b_1,b_2\in L^\times$,
so we can freely replace $t $ by any class of the form 
$t+(a_1,b)$ or $t+(a_2,b)$.  By the perfectness of the local mod $p$ Tate pairing and 
weak approximation, this means we can assume that for every valuation $ v$ 
in the support of $t$, $v(a_1)=v(a_2)=v(p)=0$.

Let $\{v_1,...,v_n\}$ denote the support of $t$.  By Hasse's theorem, we may assume 
$n\ge 2$.  We claim that we can replace $t$ by a class which 
does not include $ v_1$ or $v_2$ in its support and is supported at no more 
than one additional valuation $v_0$.  We suppose first that $a_1$ is a non-$p$th 
power in both  $L_{v_1}$ and $L_{v_2}$.  We are looking for $b$ with the 
following properties:

\begin{enumerate}

\item   $(a_1,b)_{v_1} = -t_{v_1}$
\item  $ (a_1,b)_{v_2} = -t_{v_2}$
\item $v_i(b-1)>0$ for $i\ge 3$.
\item   $b$ is congruent to 1 modulo a sufficiently high power $p^k$, so $(a_1,b)_w=0$ for all 
$w$ such that $w(p)>0$.

\end{enumerate}

To achieve these goals, we look for an element $b$ in the ring $\cO$ of 
integers of $L$ satisfying the congruence (3).  In addition, we want the 
principal ideal $(b)$ to factor as $\primep_1^{k_1} \primep_2^{k_2} \primeq$, where $\primep_i$ is 
the prime ideal of $\cO$ associated to $v_i$,   $\primeq$ 
is any prime ideal not among the $\primep_i$, and $k_1,k_2 $ are positive integers 
determined mod $p$ by conditions (1) and (2) respectively.  Note that 
because $v(a_1)=v(a_2)=0$, the symbols $(a_1,b)_{v_i}$ for $i=1,2$ depend only 
on  $v(b)$ and of course the image of $a_1$ in the residue fields of $\primep_1$ and 
$\primep_2$ respectively.  By conditions (3) and (4), $(a_1,b)_{v_i} = 0$ for $i\ge 3$
and $(a_1,b)_w=0$ for every $w$ such that $w(p)>0$.  Thus $(a_1,b)_v\neq 0$ implies
$v = v_1$, $v=v_2$, or $v=v_0$ is associated to $\primeq$.  Thus, $t+(a_1,b)$
is supported only at $v_3,\ldots,v_n$ and possibly $v_0$.

To show that $b$ exists, we use the Chinese remainder theorem to find
an element
$$b_0\in (\primep_1^{k_1}\setminus \primep_1^{k_1+1}) \cap (\primep_2^{k_2}\setminus \primep_2^{k_2+1})\cap (1+(p^k))\cap \bigcap_{i=3}^n (1+\primep_i) .$$
Let
$$\ideala = p^k\primep_1^{k_1+1}\primep_2^{k_2+1}\primep_3\cdots\primep_n.$$
We claim that there exists $b\in b_0+\ideala$ such that 
$(b)=\primep_1^{k_1}\primep_2^{k_2}\primeq$, with $\primeq$ prime.
Indeed, the condition on a prime ideal $\primeq$ which is prime to $\ideala$ that 
$\primep_1^{k_1}\primep_2^{k_2}\primeq$ is a principal ideal with a generator $b$ satisfying
conditions (3) and (4) is a non-empty condition in the ray class group of $\ideala$, so
by the Chebotarev density theorem applied to the ray class field of $\ideala$, there are infinitely
many possible prime ideals $\primeq$.

If $b$ exists, we replace $t$ with $t+(a_1,b) $ and reduce the size 
of its support, so we win by induction on $n$.
This argument assumes that $a_1$ fails to be a $p$th power in $L_{v_i}$ for 
two distinct values $ i$.  By the same reasoning, we succeed if $a_2$ fails to be a $p$th power in $L_{v_i}$
for two distinct values of $i$. Since  $a_1$ 
and $a_2$ cannot both be $p$th powers in $ L_{v_i}$ for any i, the only 
remaining possibility is that $n=2$ and (possibly renumbering) $a_i$ is a 
pth power in $L_{v_i}$ but not in $L_{v_{3-i}}$ for $ i=1,2$.  In this case, 
$a_1 a_2$ fails to be a $p$th power in both $L_{v_1} $ and $L_{v_2}$, and we look 
for $b$ as above but with $a_1$ replaced by $a_1 a_2$ in conditions (1)--(3).

We conclude that $t$ can be written $x_1 \cup \delta b_1 + x_2 \cup \delta b_2$.
Applying $i$ to both sides, the theorem follows.

 \end{proof}

  \begin{rem}
  The proof simplifies if we assume that $\mu_{p^2} \subset L$, in particular in that case we do not need Lemma \ref{fix t}, and as in the local case of Proposition \ref{solve local} we could then work with mod $p$ cohomology to prove Theorem \ref{key}, rather than with the images of the maps $i$ and $\pi$ that relate mod $p$ and mod $p^2$ cohomology. 
  \end{rem}

\section{ Three dimensional Heisenberg  representations }\label{Heisenberg}

 \begin{define}\label{def:Heisenberg}
 For $R$ any $\Z_p$-algebra we define a  ($R$-valued) Heisenberg representation to be a continuous homomorphism $\rho: G_L \to \GL_3(R)$ which takes values in the unitriangular matrices: for $\sigma \in G_L$, $\rho(\sigma)$ is thus
 of the form  \[ \left( \begin{matrix} 1 & \ast & \ast \\ 0 & 1 & \ast \\ 0 & 0 & 1 \end{matrix} \right).\]     
 \end{define}

  We  now assume that $p>2$ for  and the fields $L$ we consider in this section have $\mu_{p^2} \subset L$. Thus $\mu_{p^2}^{\otimes n}=\Z/p^2\Z$ as a $G_L$-module for all $n \in \Z$.

 We consider the case when $\rhobar:G_L \to \GL_3(\F_p)$ is a  Heisenberg representation.  The splitting field $K$ of $\rhobar$ over $L$ has  a subextension $K'/L$ with Galois group  a quotient of $\Z/p\Z \times \Z/p\Z$, and thus $K'= L( x_i^{1/p} ) $ for $i=1,2$ with $x_i \in L^\times/(L^\times)^p$.   We can also think of $x_i$ as elements in
 $H^1(L,\mu_p)$.
 
 \begin{lem}\label{red} $ $
 
 \begin{enumerate}
 
 \item  Under the cup product map $H^1(L,\mu_p) \times H^1(L,\mu_p) \to H^2(L,\mu_p^{\otimes 2}),$  $x_1 \cup x_2=0.$
 
 \item  The conjugacy classes of lifts of the homomorphism $G_L \to \Z/p\times \Z/p$,  induced by  the representation $\rhobar$  to $H$,  to mod $p$   Heisenberg representations  with values in $H$ are in the bijection with 
 $H^1(L,\mu_p^{\otimes 2})$.
 \end{enumerate}
 
 \end{lem}
 
\begin{proof}

 To see part (1),   observe that the Heisenberg group $H$  of order $p^3$ sits in an exact sequence \[ 0 \to \Z/p\Z \to H \to \Z/p\Z \times \Z/p\Z \to 0.\]
 The class in $H^2(\Z/p\Z \times \Z/p\Z ,\Z/p\Z)$ of this extension is in the one dimensional image of  
 $$H^1(\Z/p\Z,\Z/p\Z) \times H^1(\Z/p\Z,\Z/p\Z) \to H^2(\Z/p\Z \times \Z/p\Z,\Z/p\Z).$$
 Thus the fact that the extension  $K'= L( x_i^{1/p})  $ of $L$ with Galois group a quotient of  $\Z/p\Z \times \Z/p\Z$  embeds in  {\it some} $K$ with $\gal(K/L)$ a subgroup of $H$  is equivalent to   $x_1 \cup x_2=0$,   under the cup product map $H^1(L,\mu_p) \times H^1(L,\mu_p) \to H^2(L,\mu_p^{\otimes 2}),$

 Part (2) is  standard.
 \end{proof}

\begin{lem}\label{red1} $ $
The existence of a   mod $p^2$  Heisenberg representation  that reduces mod $p$ to  $\rhobar$     is equivalent to the existence of elements $\tilde x_i \in H^1(L,\mu_{p^2})$ 
lifting $x_i$ such that under the cup product map 
$H^1(L,\mu_{p^2}) \times H^1(L,\mu_{p^2}) \to H^2(L,\mu_{p^2}^{\otimes 2})$, \[\tilde x_1 \cup  \tilde x_2=0.\]

\end{lem}

\begin{proof}
 The reduction mod $p$ map $H^1(G_L,\mu_{p^2}) \to H^1(G_L,\mu_p)$ is surjective by Kummer theory. Thus the $x_i$'s indeed lift to mod $p^2$ Kummer classes.
If the lifts can be chosen to have cup product zero, then the homomorphism
 $G_L \to \Z/p^2\Z \times \Z/p^2\Z $  arising from $\gal(L(\tilde x_i^{1/p^2})/L)$ lifts to a  mod $p^2$ Heisenberg  representation $\rho'$.  The converse is also true. Further the lifts  of  $G_L \to ( \Z/p^2\Z \times \Z/p^2\Z) $ to   mod $p^2$ Heisenberg representations differ by elements of $H^1(L,\mu_{p^2}^{\otimes 2})=H^1(L,\mu_{p^2})$.

 The reduction mod $p$ of $\rho'$ will differ from $\rhobar$  by a class in  $x \in H^1(L,\mu_p)$.
  Using  the surjectivity of  the map  $H^1(G_L,\mu_{p^2}) \to H^1(G_L,\mu_p)$, we see that we can modify $\rho'$ by an element  $\tilde x \in H^1(L,\mu_{p^2})$ which reduces to $x$, to get a (conjugacy class of)   Heisenberg mod $p^2$ representation $\rho$   that reduces to $\rhobar$.
\end{proof}

\begin{thm}\label{final}
Let $L$ be a non-archimedean field of characteristic 0 or a number field and assume $\mu_{p^2} \subset L$. Then:

  \begin{enumerate}
  
  \item A mod $p$ Heisenberg representation $\rhobar:G_L \to \GL_3(\F_p)$ lifts to a  mod $p^2$  Heisenberg  representation.

  \item Let $\rhobar:G_L \to \GL_3(\F_p)$ be a homomorphism  with $p>2$. Then $\rhobar$ lifts to $\rho:G_L \to \GL_3(\Z/p^2\Z)$.

  \end{enumerate}

\end{thm}

\begin{proof}
Part (1) follows from combining Proposition \ref{solve local}, Theorem \ref{key} and Lemma \ref{red1}.  For (2) note that the obstruction $\alpha$  to lifting $\rhobar$ to a $\GL_3(\Z/p^2\Z)$ representation  lies in $H^2(G_L, M_3(\F_p))$.  We want to prove that $\alpha=0$. Let $K$ be the extension of $L$ fixed by inverse image under $\rhobar$ of the Sylow $p$-subgroup of its image. Then  if $\alpha\neq 0$,  its image $\alpha'$ under the restriction map  $H^2(G_L, M_3(\F_p)) \to H^2(G_{K(\mu_p)},M_3(\F_p))$ is non-zero. By part (1) we know that $\rhobar|_{G_{K(\mu_p)}}$ lifts to a $\Z/p^2\Z$  Heisenberg representation, and thus  $\alpha'=0$.
\end{proof}


%

 \begin{rem}
  We have left open the question of lifting $\rhobar:G_L  \to \GL_3(k)$ to a representation to $\GL_3(W_2(k))$,  for $L$ a local or global field and  $k$ is a field of characteristic $p$, outside the case  when $k$ is the prime field and $\mu_{p^2} \subset L$. We hope to return to the general case in a future work.
 \end{rem}
 
\appendix

\section{Geometric lifts of mod $p$ representations  over number fields}\label{app}

The lifting  methods specific to number fields $F$ produce lifts of  {\it odd}  representations $\br:G_F \rightarrow G(k)$ with $G$ a reductive group over a finite field $k$ to characteristic 0 {\it geometric} representations $\rho: G_F \to G(\cO)$ with $\cO$ the ring of integers of a finite extension of $\Q_p$  (see \cite[ Definition 1.2]{fkp:reldef} for the definition of odd $\rhobar$ which in particular implies that $F$  is a CM field). There are two distinct methods,  as recalled in the introduction,  to produce such geometric lifts (cf.  \cite{ramakrishna02}  and \cite{khare-wintenberger:serre0}) which were first developed there  in the particular case  of odd irreducible representations $\rhobar:G_\Q \to \GL_2(k)$.   The papers  \cite{ramakrishna-hamblen},  \cite{stp:exceptional}, \cite{fkp:reldef} and \cite{fkp:reducible}   generalize  R.~Ramakrishna's method in \cite{ramakrishna02}, and \cite{blggt:potaut} generalizes the   method of  \cite{khare-wintenberger:serre0}. The innovation of \cite{klr} plays a key role in  \cite{ramakrishna-hamblen}, \cite{fkp:reldef} and \cite{fkp:reducible} to handle cases when the representation has small image (for eg.  when $G=\GL_2$, the ``doubling method'' of \cite{klr} plays a key role  in \cite{ramakrishna-hamblen}  to produce  {\it irreducible} geometric lifts  of  odd, {\it reducible} representations $\rhobar:G_\Q \to \GL_2(k)$).

 These methods lift $\br$ to  characteristic 0 representations with control on the local ramification behavior of the lifts (at all primes in the method of  \cite{khare-wintenberger:serre0}, while in  the Ramakrishna method  one has to generally allow    extra ramification at finite sets of primes chosen to kill certain dual Selmer groups). Thus this goes   beyond the mod $p^2$ lifting produced by Kummer theory in \cite{K-JNT}.  

In spite of the differences in the methods of  \cite{ramakrishna02} and \cite{khare-wintenberger:serre0}, they both rely on   certain numerics  which obtain for odd Galois representations ( ``dimension of Selmer'' $h^1_f(S,\Ad(\rhobar)) $ $\geq$ $h^1_f(S,\Ad(\rhobar)^*)$ ``dimension of dual Selmer'') that  guarantee an appropriate deformation ring has positive dimension. When these numerics fail one's bearings are lost  and the question  if there exist geometric lifts of residual representations
$\rhobar:G_K \to  G(k)$ is as yet mysterious.  To hazard a guess for the  class of $\rhobar$ which have  geometric lifts we  consider  the class of  representations  $\rhobar$ which are ``regular at infinity'' which is broader than  the class of odd representations of \cite[ Definition 1.2]{fkp:reldef}. This notion has been considered previously as the property of being $\GL$-odd  in \cite{calegari:evenNotes}.

\begin{defn}
 We say that $\rhobar:G_K \to \GL_n(k)$ is {\it  regular at infinity}  if for each real place of $v$, for the conjugacy class $c_v$  in $G_K$ of complex conjugation at $v$,   the characteristic polynomial of  $\rhobar(c_v)$  which is of the form $(X-1)^{n_{v,+}}(X+1)^{n_{v,-}}$,  has the property that  $|n_{v,+}-n_{v,-}| \leq 1$.  (For the case of $k$ of characteristic 2, we interprete this to mean that  there is no condition.) 
\end{defn}

\begin{conj}\label{mystery}
Let $K$ be a number field and $k$ a finite field of characteristic $p$. A continuous  representation $\rhobar: G_K \to \GL_n(k)$  that is regular at infinity lifts to an irreducible  geometric representation $\rho: G_K \to  \GL_n(\cO)$ with $\cO$ the ring of integers of a finite extension of $\Q_p$.

\end{conj}

 The condition that  $\rhobar$   is  regular at infinity  is  a reasonable necessary  condition to impose for it to have geometric lifts  in the light of \cite[Theorem 5.1]{calegari:even2}  and \cite[Theorem 1.1]{bao-caraiani}. We could  be bolder and  make the stronger conjecture that  $\rhobar:G_K \to \GL_n(k)$ has a  geometric lift that is  regular at  all places above $p$ (i.e., distinct Hodge-Tate weights for all places above $p$) if and only $\rhobar$ is regular at infinity.  (Note that  we have not assumed that  $\rhobar$ is irreducible as we believe that  not to be necessary,
 cf. \cite{ramakrishna-hamblen} and \cite{fkp:reducible}.)

The question is accessible for  $n=1$ (trivially), and for $n=2$ when $K$ is a totally real field (results of \cite{ramakrishna02}, \cite{khare-wintenberger:serre0},   \cite{ramakrishna-hamblen} and \cite{fkp:reducible}). Beyond these cases the conjecture  for representations $\rhobar$ that are regular at infinity,  but not odd in the sense of \cite[ Definition 1.2]{fkp:reldef},  seems to us inaccessible at the moment.

One cannot expect there  to be ``minimal lifts''  (with the best possible ramification behavior roughly speaking, see \cite[\S 3]{khare-wintenberger:serre0})  in this generality, unlike the case of odd Galois representations  in which case such minimal lifts are produced by the method of \cite{khare-wintenberger:serre0}. This seems to make the question harder, as geometric lifts of  $\rhobar$ that are regular at infinity but not odd,   even if they were to exist, would have a sporadic character, making searching for a geometric lift of  such $\rhobar$ like   looking for a needle in a haystack.

 Arguably the simplest case in which the conjecture is open is for representations  $\rhobar:G_K \to \GL_2(k)$ when $K$ is an imaginary quadratic field, like $K=\Q(i)$: note that these representations are regular at infinity (as is   any  representation $\rhobar:G_K \to \GL_n(k)$  if $K$ has no real place).  We also note that  while the existence of geometric lifts in these cases seems presently out of reach, lifting $\rhobar$ to a representation $\rho: G_K \to  \GL_2(\cO)$ which is not necessarily geometric is possible by the ``doubling  method'' of \cite{klr}. The  method  of \cite{klr} overcomes the numerical disadvantage,   $h^1_f(S,\Ad(\rhobar)) $ $<$ $h^1_f(S,\Ad(\rhobar)^*)$,  one is at if one restricts to producing geometric lifts that are unramified outside a  finite set  of primes $S$,  by means of allowing either ramification at infinitely many  primes, or dropping the condition on the lifts of being de Rham  at primes above $p$.

\bibliographystyle{amsalpha}
\bibliography{biblio.bib}

\def\cprime{$'$}
\providecommand{\bysame}{\leavevmode\hbox to3em{\hrulefill}\thinspace}
\providecommand{\MR}{\relax\ifhmode\unskip\space\fi MR }
\providecommand{\MRhref}[2]{%
  \href{http://www.ams.org/mathscinet-getitem?mr=#1}{#2}
}
\providecommand{\href}[2]{#2}
\begin{thebibliography}{BLGGT14}

\bibitem[BLGGT14]{blggt:potaut}
Thomas Barnet-Lamb, Toby Gee, David Geraghty, and Richard Taylor,
  \emph{Potential automorphy and change of weight}, Ann. of Math. (2)
  \textbf{179} (2014), no.~2, 501--609. \MR{3152941}

\bibitem[Boe03]{bockle:modp2}
Gebhard Boeckle, \emph{Lifting mod {$p$} representations to characteristics
  {$p^2$}}, J. Number Theory \textbf{101} (2003), no.~2, 310--337. \MR{1989890}

\bibitem[Cal]{calegari:evenNotes}
Frank Calegari, \emph{Even galois representations}, notes available at author's
  web page.

\bibitem[Cal12]{calegari:even2}
\bysame, \emph{Even {G}alois representations and the {F}ontaine--{M}azur
  conjecture. {II}}, J. Amer. Math. Soc. \textbf{25} (2012), no.~2, 533--554.
  \MR{2869026}

\bibitem[CPS75]{cline-parshall-scott}
Edward Cline, Brian Parshall, and Leonard Scott, \emph{Cohomology of finite
  groups of {L}ie type. {I}}, Inst. Hautes \'Etudes Sci. Publ. Math. (1975),
  no.~45, 169--191. \MR{MR0399283 (53 \#3134)}

\bibitem[FdC18]{mfcdc}
Mathieu Florence and Charles de~Clercq, \emph{Lifting low-dimensional local
  systems}, arXiv : 1812.08068 (2018).

\bibitem[FKP19]{fkp:reldef}
Najmuddin {Fakhruddin}, Chandrashekhar {Khare}, and Stefan {Patrikis},
  \emph{{Relative deformation theory and lifting irreducible Galois
  representations}}, arXiv e-prints (2019), arXiv:1904.02374.

\bibitem[FKP20]{fkp:reducible}
Najmuddin Fakhruddin, Chandrashekhar Khare, and Stefan Patrikis, \emph{Lifting
  and modularity of reducible mod $ p$ galois representations}, in preparation
  (2020).

\bibitem[HR08]{ramakrishna-hamblen}
Spencer Hamblen and Ravi Ramakrishna, \emph{Deformations of certain reducible
  {G}alois representations. {II}}, Amer. J. Math. \textbf{130} (2008), no.~4,
  913--944. \MR{2427004}

\bibitem[Ill05]{serre:illusie}
Luc Illusie, \emph{Grothendieck's existence theorem in formal geometry. {W}ith
  a letter of {J}ean-{P}ierre {S}erre}, Math. Surveys Monogr., Fundamental
  algebraic geometry, A.M.S. \textbf{123} (2005), 179–233. \MR{MR 2223409}

\bibitem[JP76]{jones-parshall}
Wayne Jones and Brian Parshall, \emph{On the {$1$}-cohomology of finite groups
  of {L}ie type}, Proceedings of the {C}onference on {F}inite {G}roups ({U}niv.
  {U}tah, {P}ark {C}ity, {U}tah, 1975), 1976, pp.~313--328. \MR{0404470}

\bibitem[Kha97]{K-JNT}
Chandrashekhar Khare, \emph{Base change, lifting, and {S}erre's conjecture}, J.
  of Number Theory \textbf{63} (1997), no.~2, 387--395.

\bibitem[KLR05]{klr}
Chandrashekhar Khare, Michael Larsen, and Ravi Ramakrishna, \emph{Constructing
  semisimple {$p$}-adic {G}alois representations with prescribed properties},
  Amer. J. Math. \textbf{127} (2005), no.~4, 709--734. \MR{2154368}

\bibitem[KW09a]{khare-wintenberger:serre0}
Chandrashekhar Khare and Jean-Pierre Wintenberger, \emph{On {S}erre's
  conjecture for 2-dimensional mod {$p$} representations of {${\rm
  Gal}(\overline{\Bbb Q}/\Bbb Q)$}}, Ann. of Math. (2) \textbf{169} (2009),
  no.~1, 229--253. \MR{2480604}

\bibitem[KW09b]{khare-wintenberger:serre1}
\bysame, \emph{Serre's modularity conjecture. {I}}, Invent. Math. \textbf{178}
  (2009), no.~3, 485--504. \MR{2551763 (2010k:11087)}

\bibitem[Le90]{lurye}
B.~B. Lur\cprime~e, \emph{Universally solvable embedding problems}, vol. 183,
  1990, Translated in Proc. Steklov Inst. Math. {{\bf{1}}991}, no. 4, 141--147,
  Galois theory, rings, algebraic groups and their applications (Russian),
  pp.~121--126, 225--227. \MR{1092023}

\bibitem[LHC16]{bao-caraiani}
Bao~V. Le~Hung and Ana Caraiani, \emph{On the image of complex conjugation in
  certain galois representations}, Compositio Mathematica \textbf{152} (2016),
  no.~37, 1476–1488. \MR{93d:11059}

\bibitem[Pat16]{stp:exceptional}
Stefan Patrikis, \emph{Deformations of galois representations and exceptional
  monodromy}, Inventiones mathematicae \textbf{205} (2016), no.~2, 269--336.

\bibitem[Ram02]{ramakrishna02}
Ravi Ramakrishna, \emph{Deforming {G}alois representations and the conjectures
  of {S}erre and {F}ontaine-{M}azur}, Ann. of Math. (2) \textbf{156} (2002),
  no.~1, 115--154. \MR{MR1935843 (2003k:11092)}

\bibitem[Ser87]{serre:conjectures}
Jean-Pierre Serre, \emph{Sur les repr\'esentations modulaires de degr\'e
  \protect{$2$} de \protect{${\rm{G}al}(\overline{\bf {Q}}/{\bf {Q}})$}}, Duke
  Math. J. \textbf{54} (1987), no.~1, 179--230.

\bibitem[Ser98]{serre:ladic}
\bysame, \emph{Abelian \protect{$\ell$}-adic representations and elliptic
  curves}, A\thinspace{}K Peters Ltd., Wellesley, MA, 1998, With the
  collaboration of Willem Kuyk and John Labute, Revised reprint of the 1968
  original.

\bibitem[Ser00]{serre:oeuvres}
\bysame, \emph{{\OE}uvres. {C}ollected papers. {IV}}, Springer-Verlag, Berlin,
  2000, 1985--1998. \MR{2001e:01037}

\bibitem[Vas03]{vasiu-surj}
Adrian Vasiu, \emph{Surjectivity criteria for {$p$}-adic representations. {I}},
  Manuscripta Math. \textbf{112} (2003), no.~3, 325--355. \MR{2067042}

\bibitem[Wil95]{wiles:fermat}
Andrew Wiles, \emph{Modular elliptic curves and \protect{F}ermat's last
  theorem}, Ann. of Math. (2) \textbf{141} (1995), no.~3, 443--551.

\end{thebibliography}

\end{document}